\documentclass[reqno]{amsart}

\usepackage[T1]{fontenc}
\usepackage{lmodern}
\usepackage{amsmath,amssymb,amsthm}
\usepackage{array}
\usepackage{enumitem}
\usepackage{microtype}
\usepackage{cite}
\usepackage{needspace}
\usepackage[hidelinks]{hyperref}
\usepackage{tikz}
\usetikzlibrary{positioning}

\makeatletter
\@namedef{subjclassname@2020}{\textup{2020} Mathematics Subject Classification}
\makeatother

\numberwithin{equation}{section}

\newcommand{\B}{\mathcal B_2^1}
\newcommand{\Id}{\operatorname{Id}}
\newcommand{\con}{\operatorname{con}}

\newtheorem{theorem}{Theorem}[section]
\newtheorem{lemma}[theorem]{Lemma}
\newtheorem{proposition}[theorem]{Proposition}
\theoremstyle{remark}

\begin{document}

\title[The six-element Brandt ai-semiring]
{The six-element additively idempotent semiring $B_2^1$ is inherently nonfinitely based}

\author[M. M. Ren]{Miao Miao Ren}
\address{School of Mathematics\\
Northwest University\\Xi'an 710127\\Shaanxi\\P.R. China}
\email{miaomiaoren@yeah.net}

\author[M. Y. Yue]{Meng Ya Yue}
\address{School of Mathematics\\
Northwest University\\Xi'an 710127\\Shaanxi\\P.R. China}
\email{myayue@yeah.net}

\author[Y. L. Zhou]{Yi Lin Zhou}
\address{School of Mathematics\\
Northwest University\\Xi'an 710127\\Shaanxi\\P.R. China}
\email{zhou990619@163.com}


\hypersetup{
  pdftitle={The six-element Brandt ai-semiring is inherently nonfinitely based},
  pdfauthor={Miao Miao Ren, Meng Ya Yue, and Yi Lin Zhou},
  pdfsubject={Finite basis problem for additively idempotent semirings},
  pdfkeywords={additively idempotent semiring, Brandt monoid, finite basis problem, Zimin word}
}

\begin{abstract}
We show that the six-element additively idempotent semiring $B_2^1$ is inherently
nonfinitely based. More precisely, for every positive integer $n$, we
construct an infinite additively idempotent semiring, generated by at most $(6n+8)^2$
elements, that satisfies every $n$-variable identity of $B_2^1$.
\end{abstract}

\maketitle

\section{Introduction}\label{sec:introduction}

An \emph{additively idempotent semiring}, or \emph{ai-semiring} for
short, is an algebra $(S,+,\cdot)$ such that $(S,+)$ is a commutative
idempotent semigroup, $(S,\cdot)$ is a semigroup, and multiplication
distributes over addition from both sides:
\[
x(y+z)\approx xy+xz
\quad\text{and}\quad
(x+y)z\approx xz+yz.
\]
We work throughout in the signature consisting of the two binary
operations $+$ and $\cdot$. General background on semirings and their
applications can be found in~\cite{Glazek2001,Golan1992}.

A finite algebra is \emph{inherently nonfinitely based} if it belongs to
no finitely based locally finite variety of the same signature.  This
notion arose from the study of hereditary finite basability of finite
semigroups; see Sapir and Volkov~\cite{SapirVolkov1988}.  In the semiring setting, Dolinka constructed finite nonfinitely
based examples~\cite{Dolinka2007} and subsequently obtained a general
class of inherently nonfinitely based semirings~\cite{Dolinka2009}.
The finite basis problem for ai-semirings has since been studied from
several directions, including flat extensions, power semirings,
endomorphism semirings and small finite ai-semirings; see, for example,
\cite{DGV2024,dgv,JacksonRenZhao2022,RenJacksonZhaoLei2023,
WuRenZhao2024,ZhaoEtAl2020,ZhaoWu2025}.

Let $E_{ij}$ denote the $2\times2$ Boolean matrix unit. The
six-element Brandt monoid is
\[
 B_2^1=\{0,1,c,d,cd,dc\},
 \qquad
 c=E_{12},\quad d=E_{21},
\]
where $cd=E_{11}$, $dc=E_{22}$, and multiplication is Boolean matrix
multiplication. It admits a unique compatible semilattice addition. We
denote the resulting ai-semiring by $\B$; its addition and the notation
for its elements are described explicitly in Section~2. The
multiplicative reduct of $\B$ is the six-element Brandt monoid.  Perkins~\cite{Perkins1969} proved that this
monoid is nonfinitely based.  Volkov~\cite{Volkov2021} proved that the
identities of $\B$ admit no basis involving only a bounded number of
variables, and therefore that $\B$ is nonfinitely based.  The same
finite basis problem was also treated by Jackson, Ren and
Zhao~\cite{JacksonRenZhao2022}.  The question whether $\B$ is
inherently nonfinitely based was formulated there and was restated by
Dolinka, Gusev and Volkov~\cite{dgv}.

Our main result answers this question affirmatively.  For every
$n\geq1$ we construct an infinite ai-semiring $T_n$, generated by at
most $(6n+8)^2$ elements, that satisfies every identity of $\B$ in at
most $n$ variables.  The multiplicative reduct of $T_n$ is based on a
factor semigroup used by Auinger, Dolinka and Volkov~\cite{ADV2012}.
The proof combines the minimality of Zimin words with a return property
of the factor language and a prefix-code reduction.

For an algebra $S$, let $\Id_{\leq n}(S)$ denote the set of all
identities of $S$ whose two sides together involve at most $n$
distinct variables.  The main theorem is stated in Section~7 after the
necessary construction has been established.

\section{Preliminaries and the six-element algebra}

In this section, we fix the notation and terminology used throughout
the paper. We refer to~\cite{BurrisSankappanavar1981,KurilPolak2005}
for the standard universal-algebraic and semilattice-ordered semigroup
background.

Let $X$ be a countably infinite alphabet. Let $X^+$ be the free
semigroup of nonempty words,
and let $X^*=X^+\cup\{\varepsilon\}$ be the free monoid. A \emph{factor}
always means a contiguous factor. Word substitutions are nonerasing:
each variable is replaced by a nonempty word. Empty words are permitted
as contexts only, never as terms or substitution images.

The \emph{content} $\con(P)$ of a term $P$ is its set of variables.
For terms or words $P,Q$ and an ai-semiring $S$, put
\[
 S\models P\preccurlyeq Q
 \quad\Longleftrightarrow\quad S\models P+Q\approx Q.
\]
For elements the same convention defines the additive order.
For example, $u(s+t)=ut$ follows from $s+t=t$, so multiplication on
either side preserves this order.

The elements of $B_2^1$ may be represented by the Boolean matrices
\[
\begin{array}{c@{\qquad}c@{\qquad}c}
0=\begin{pmatrix}0&0\\0&0\end{pmatrix}&
1=\begin{pmatrix}1&0\\0&1\end{pmatrix}&
c=\begin{pmatrix}0&1\\0&0\end{pmatrix}\\[7pt]
d=\begin{pmatrix}0&0\\1&0\end{pmatrix}&
cd=\begin{pmatrix}1&0\\0&0\end{pmatrix}&
dc=\begin{pmatrix}0&0\\0&1\end{pmatrix}.
\end{array}
\]
Multiplication is Boolean matrix multiplication, while addition is the
intersection of matrix supports. Thus $0$ is the zero of multiplication
and the greatest element of the additive order, whereas $1$ is the
identity of multiplication and is not an additive identity. As usual,
\[
 x\leq y\quad\Longleftrightarrow\quad x+y=y.
\]
Consequently, the additive order is reverse inclusion of supports and
is displayed in Figure~\ref{fig:additive-order}.

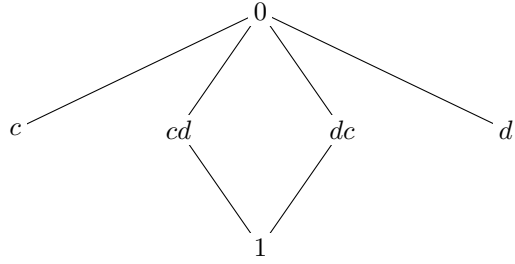
\begin{figure}[ht]
\centering
\begin{tikzpicture}[
  x=1.2cm,y=1.15cm,
  element/.style={font=\normalsize,inner sep=2pt,fill=white}
]
\coordinate (top) at (0,2.7);
\coordinate (bot) at (0,0);
\coordinate (c0)  at (-2.7,1.35);
\coordinate (cd0) at (-0.9,1.35);
\coordinate (dc0) at (0.9,1.35);
\coordinate (d0)  at (2.7,1.35);
\draw (bot)--(cd0) (bot)--(dc0);
\draw (c0)--(top) (cd0)--(top) (dc0)--(top) (d0)--(top);
\node[element] at (top) {$0$};
\node[element] at (c0) {$c$};
\node[element] at (cd0) {$cd$};
\node[element] at (dc0) {$dc$};
\node[element] at (d0) {$d$};
\node[element] at (bot) {$1$};
\end{tikzpicture}
\caption{The additive order of $\mathcal B_2^1$.}
\label{fig:additive-order}
\end{figure}

For completeness, the two operations are recorded in
Table~\ref{tab:cayley}. In each table, the first column specifies the
left operand and the first row the right operand. Thus the entry in row
$x$ and column $y$ is $x+y$ in the left-hand table and $xy$ in the
right-hand table.
\begin{table}[ht]
\caption{The Cayley tables of $\mathcal B_2^1$}
\label{tab:cayley}
\centering
\(
\begin{array}{c|cccccc}
+&0&1&c&d&cd&dc\\\hline
0&0&0&0&0&0&0\\
1&0&1&0&0&cd&dc\\
c&0&0&c&0&0&0\\
d&0&0&0&d&0&0\\
cd&0&cd&0&0&cd&0\\
dc&0&dc&0&0&0&dc
\end{array}
\qquad
\begin{array}{c|cccccc}
\cdot&0&1&c&d&cd&dc\\\hline
0&0&0&0&0&0&0\\
1&0&1&c&d&cd&dc\\
c&0&c&0&cd&0&c\\
d&0&d&dc&0&d&0\\
cd&0&cd&c&0&cd&0\\
dc&0&dc&0&d&0&dc
\end{array}
\)
\end{table}

\Needspace{8\baselineskip}
\begin{lemma}\label{lem:matrix}
The operations above make $\B$ an ai-semiring. Its additive order is
the order in Figure~\ref{fig:additive-order}.
\end{lemma}
\begin{proof}
Boolean matrix multiplication is associative, and intersection of
supports is associative, commutative and idempotent. Since the six
supports are graphs of partial bijections, Boolean matrix multiplication
distributes over their intersection on both sides. The displayed tables
show that both operations are closed on $B_2^1$. Finally,
\[
 x+y=y
 \quad\Longleftrightarrow\quad
 \operatorname{supp}(y)\subseteq\operatorname{supp}(x),
\]
which gives precisely the order in Figure~\ref{fig:additive-order}.
\end{proof}

\begin{lemma}\label{lem:normal}
Every term in the signature $\{+,\cdot\}$ is, in every ai-semiring,
equal to a finite nonempty sum of nonempty words. The union of the
contents of these words is the content of the original term.
\end{lemma}
\begin{proof}
Use induction on the term. A variable gives the singleton set consisting
of its one-letter word. For $P+Q$, take the union of the two finite
nonempty sets obtained inductively. For $PQ$, take all concatenations
$uv$, with $u$ from the first set and $v$ from the second. Both sets
remain finite and nonempty; associativity, commutativity and idempotence
justify the union step, and the two distributive laws justify the
concatenation step. No empty word is introduced. The statement about
content follows in each case by taking unions of the variable sets.
\end{proof}

\begin{lemma}\label{lem:regular}
Every identity of $\B$ is regular: its two sides have the same content.
\end{lemma}
\begin{proof}
The subset $\{0,1\}$ is closed under both operations. Either operation
returns $1$ exactly when both operands are $1$. For every term $P$ and
assignment $\nu$ into this subset, structural induction gives
\[
 \nu(P)=1\quad\Longleftrightarrow\quad
 \nu(x)=1\text{ for every }x\in\con(P).
\]
The claim is immediate for a variable. If $P=U+V$ or $P=UV$, its value
is $1$ exactly when both subterms have value $1$; their induction
hypotheses yield precisely the condition for
$\con(P)=\con(U)\cup\con(V)$.
If $x\in\con(P)\setminus\con(Q)$, assign $x$ to $0$ and every other
variable to $1$. The values of $P,Q$ are $0,1$, respectively, so
$P\approx Q$ cannot be an identity. Interchanging the two terms rules
out the other set difference.
\end{proof}

\section{Minimality of Zimin words}

A nonempty word $w$ is \emph{minimal} for an ai-semiring $S$ if
$S\models v\preccurlyeq w$ for a nonempty word $v$ implies the literal
word equality $v=w$. Define the Zimin words by
\[
 Z_1=x_1,\qquad Z_{m+1}=Z_mx_{m+1}Z_m\quad(m\ge1).
\]

\Needspace{4\baselineskip}
\begin{lemma}[{\cite[Lemma~2]{dgv}}]\label{lem:zimin}
Every Zimin word is minimal for $\B$.
\end{lemma}

\begin{lemma}\label{lem:factor}
In any ai-semiring, every nonempty factor of a minimal word is minimal.
\end{lemma}
\begin{proof}
Write the minimal word as $W=\alpha t\beta$, with $t$ nonempty
and $\alpha,\beta$ possibly empty. If $S\models s\preccurlyeq t$,
distributivity gives $S\models us\preccurlyeq ut$ and
$S\models su\preccurlyeq tu$ for a nonempty word $u$; for example,
$us+ut=u(s+t)=ut$. Apply these implications on each nonempty
context side, omitting multiplication on an empty side. It follows
that $S\models\alpha s\beta\preccurlyeq\alpha t\beta$.
Minimality gives the literal equality
$\alpha s\beta=\alpha t\beta$. Cancellation in the free monoid,
not in the algebra $S$, gives $s=t$.
\end{proof}

\section{Sapir's factor semigroup and a return property}

Fix $k\ge1$, put $r=6k+2$, and let
\[
 A_k=\{a_{ij}:1\le i,j\le r\}.
\]
Define a morphism $\gamma:A_k^+\longrightarrow A_k^+$ by
\begin{equation}\label{eq:gamma}
 \gamma(a_{ij})=a_{i1}a_{j2}a_{i3}a_{j4}\cdots a_{i,r-1}a_{j,r}.
\end{equation}
Let $L_k$ be the set of all nonempty factors of $\gamma^m(a_{11})$
for $m\ge1$. It is a \emph{factorial} language: every nonempty
factor of one of its words is again in it. Adjoin a new symbol $0$
and define
\begin{equation}\label{eq:product}
 p\cdot q=\begin{cases}
 pq,&p,q\in L_k\text{ and }pq\in L_k,\\
 0,&\text{otherwise}.
 \end{cases}
\end{equation}
The resulting semigroup is denoted by $V_k^0$.

\begin{lemma}\label{lem:factorproduct}
The product in~\eqref{eq:product} is associative. A finite nonempty
product of nonzero words is their literal concatenation if this belongs
to $L_k$, and is zero otherwise. Any product with a zero factor is zero.
\end{lemma}
\begin{proof}
If $p,q,s$ are nonzero and $pqs\in L_k$, factoriality implies
$pq,qs\in L_k$, so both bracketings equal $pqs$. If $pqs\notin L_k$,
each bracketing is zero: either its first multiplication is already
zero, or the second multiplication rejects the complete word.
If an operand is zero both bracketings are zero. This proves
associativity. Induction on the number of factors proves the rest:
factoriality ensures every partial concatenation of an allowed complete
word is allowed, while a forbidden factor cannot occur in an allowed
complete word.
\end{proof}

\begin{proposition}[Sapir; Auinger--Dolinka--Volkov]\label{prop:sapir}
For the parameter $k$, alphabet, morphism and factor semigroup just
defined, put $X_k=\{x_1,\ldots,x_k\}$. Let $u,v\in X_k^+$.
If a semigroup homomorphism $\phi:X_k^+\longrightarrow V_k^0$
satisfies $\phi(u)\ne\phi(v)$, then the single identity $u\approx v$
entails, in the class of ordinary semigroups, an identity
$Z_{k+1}\approx q$ with $q\ne Z_{k+1}$ as literal words.
\end{proposition}

This is Proposition~2.5 of Auinger, Dolinka and Volkov
\cite{ADV2012}, stated here in the notation used below.  In particular,
the proposition applies to every failed identity involving at most $k$
variables, after renaming its variables into $X_k$.

\begin{lemma}\label{lem:return}
For each $c\in A_k$, the word $\gamma^3(c)$ contains every letter
of $A_k$. For every $w\in L_k$ there is a nonempty $d\in L_k$
such that $wdw\in L_k$.
\end{lemma}
\begin{proof}
For each $p\in\{1,\ldots,r\}$, the word $\gamma(c)$ contains
a letter whose second index is $p$. Its image contains $a_{pq}$
at every even position $q$. Hence $\gamma^2(c)$ contains every
even-column letter, in particular $a_{p2}$ for every $p$.
The image of $a_{p2}$ contains $a_{pq}$ at every odd position $q$,
so $\gamma^3(c)$ contains every odd-column letter. To check the
even columns as well, choose any letter $e$ occurring in
$\gamma(c)$. The factor $\gamma^2(e)$ occurs in $\gamma^3(c)$,
and the preceding even-column argument applies to $e$. Thus all
letters occur in $\gamma^3(c)$.

Choose $m\ge1$ so that $w$ is a factor of $\gamma^m(a_{11})$.
Since $a_{11}$ occurs in $\gamma^3(c)$ for every $c$, the word
$\gamma^{m+3}(c)$ contains $\gamma^m(a_{11})$, hence $w$.
Write $\gamma(a_{11})=c_1\cdots c_r$. Then
\[
 \gamma^{m+4}(a_{11})
 =\gamma^{m+3}(c_1)\cdots\gamma^{m+3}(c_r).
\]
As $r\ge8$, the first three blocks exist. Choose an occurrence of
$w$ entirely within the first block and another entirely within the
third. They are disjoint. The intervening word $d$ contains the
whole second block, of length $r^{m+3}>0$, so $d$ is nonempty.
The factor spanning the two occurrences is $wdw$. Both it and its
nonempty factor $d$ lie in $L_k$. This argument holds for every
$k\ge1$ and every length of $w$.
\end{proof}

On $L_k\cup\{0\}$ define the flat addition
\begin{equation}\label{eq:flat}
 p+p=p,\qquad p+q=0\quad\text{if }p\ne q,
\end{equation}
and write $F_k=(L_k\cup\{0\},+,\cdot)$.

\begin{lemma}\label{lem:flat}
The algebra $F_k$ is an infinite ai-semiring generated by the $r^2$
letters of $A_k$.
\end{lemma}
\begin{proof}
Multiplication is associative by Lemma~\ref{lem:factorproduct}.
Addition is commutative and idempotent by definition, and
$0+p=0$ for every $p$, including $p=0$.
For associativity, if $p=q=s$, both bracketings give that common
value. Otherwise, suppose first that $p\ne q$; then
$(p+q)+s=0$. If $q\ne s$, the other bracketing is $p+0=0$;
if $q=s$, it is $p+q=0$. The case $p=q$ but $q\ne s$
follows by the same direct computation. This covers the entire
carrier, including zero.

For left distributivity, if $p=q$, then
$s(p+q)=sp=sp+sp$. If $p\ne q$, then $s(p+q)=s0=0$.
Equality $sp=sq\ne0$ would force $s,p,q$ to be nonzero words,
and literal left cancellation would give $p=q$, a contradiction.
Therefore $sp,sq$ are either distinct, with flat sum zero, or both
zero, again with sum zero. This proves left distributivity in all
cases. For right distributivity, use the same two cases and literal
right cancellation if $ps=qs\ne0$.

Every letter of $A_k$ belongs to $L_k$ by Lemma~\ref{lem:return}.
Every word of $L_k$ is the product of its letters, because each
prefix remains in the factorial language. Two distinct letters exist
and have sum zero. Thus these $r^2$ letters generate the whole
algebra without a named constant. Finally, $\gamma^m(a_{11})$
itself belongs to $L_k$ and has length $r^m$. These lengths are
unbounded, proving that $F_k$ is infinite.
\end{proof}

\section{A guarded identity and transfer at nonzero values}

\begin{lemma}[The guarded identity]\label{lem:guarded}
Let $s,t$ be nonempty words with
$|\con(s)\cup\con(t)|\le d$ and
$\B\models s\preccurlyeq t$. If $z$ is a fresh variable, then
\begin{equation}\label{eq:guarded}
 V_k^0\models tzt\approx szt\qquad\text{whenever }k\ge d+1.
\end{equation}
\end{lemma}
\begin{proof}
We first show that the single semigroup identity
$\eta:\ tzt\approx szt$ entails no nontrivial identity
$Z_m\approx q$, for any $m\ge1$.

Here is an explicit justification of the replacement form of this
entailment. Let $\Omega$ be a countably infinite alphabet containing
all variables concerned. On $\Omega^+$ allow elementary moves
\[
 \alpha\theta(tzt)\beta
 \ \longleftrightarrow\ \alpha\theta(szt)\beta,
 \qquad \alpha,\beta\in\Omega^*,
\]
where $\theta$ is nonerasing. Two words are related if they are
connected by a finite sequence of such moves, including a sequence
of length zero. This is an equivalence relation stable under left
and right concatenation, so it is a semigroup congruence.
Its quotient satisfies $\eta$: for any assignment of the finitely
many variables in $\eta$ to quotient classes, choose a nonempty
word representative of each class. These representatives define
$\theta$, and one elementary move identifies the two evaluations.
Consequently, if $\eta$ entails $Z_m\approx q$ in all semigroups,
it holds in this quotient under the assignment sending a variable to
its one-letter class. Thus a finite sequence of moves connects
$Z_m$ to $q$.

Consider any elementary move starting at the literal word $Z_m$.
In one direction, $\theta(tzt)$ is a factor of $Z_m$; in the
other, $\theta(szt)$ is a factor of $Z_m$. Both contain the same
nonempty suffix $\theta(t)$. Lemmas~\ref{lem:zimin}
and~\ref{lem:factor} make this suffix minimal for $\B$.
Substitution into the original inequality gives
\[
 \B\models\theta(s)\preccurlyeq\theta(t).
\]
Minimality forces literal equality $\theta(s)=\theta(t)$.
The two possible replaced words are therefore literally equal, and
the move leaves $Z_m$ unchanged. Induction along a finite sequence
of moves shows that no such sequence leaves $Z_m$. Hence the
entailed identity cannot have $q\ne Z_m$.

The union of the variable sets of $tzt$ and $szt$ has size
at most $d+1\le k$. If~\eqref{eq:guarded} failed, then
Proposition~\ref{prop:sapir} would give a nontrivial entailed
identity $Z_{k+1}\approx q$, contrary to the preceding argument.
\end{proof}

\begin{lemma}[Transfer at a nonzero value]\label{lem:transfer}
Suppose $\B\models s\preccurlyeq t$, where $s,t\in Y^+$ and
$|Y|\le k-1$. Let $f:Y^+\longrightarrow A_k^+$ be a morphism
with $f(y)\in L_k$ for every $y\in Y$. If $f(t)\in L_k$,
then $f(s)=f(t)$ as literal words.
\end{lemma}
\begin{proof}
Put $w=f(t)$ and use Lemma~\ref{lem:return} to choose a nonempty
$d\in L_k$ with $wdw\in L_k$. By Lemma~\ref{lem:guarded},
$V_k^0\models tzt\approx szt$. Assign $y$ to the element
$f(y)$ and $z$ to $d$. Lemma~\ref{lem:factorproduct} gives the
nonzero value $wdw$ on the left. Equality forces the right side
to be nonzero as well. Only now identify its value, by the same
lemma, with the literal concatenation $f(s)dw$. Hence
\[
 f(s)dw=wdw.
\]
Right cancellation in the free monoid $A_k^*$ gives $f(s)=w$.
No cancellation property of $V_k^0$ is used. The conclusion also
shows that $f(s)\in L_k$.
\end{proof}

\section{Prefix-code reduction}

A set $D$ of nonempty words is a \emph{prefix code} if no member
is a proper prefix of another member. Write $D^+$ for all
nonempty concatenations of members of $D$.

\begin{lemma}\label{lem:prefix}
Let $C$ be a finite nonempty set of nonempty words in a factorial
language $L$. There is a finite nonempty prefix code $D\subseteq L$
such that $|D|\le|C|$ and $C\subseteq D^+$.
\end{lemma}
\begin{proof}
Start with the current set $S=C$. If distinct members have the
form $u$ and $uv$, with $v\ne\varepsilon$, replace $S$ by
\[
 S'=(S\setminus\{uv\})\cup\{v\}.
\]
The word $u$ remains in $S'$, since $u\ne uv$. Each old
generator belongs to $(S')^+$: only $uv$ was removed, and it
is the product of $u$ and $v$. Thus $S^+\subseteq(S')^+$,
so the invariant $C\subseteq S^+$ is preserved.

The current set stays nonempty, and its cardinality never increases.
Every new word is a nonempty suffix of an old member and therefore
still lies in $L$. The positive integer measure
$\sum_{w\in S}|w|$ strictly decreases. If $v$ is not already
in $S\setminus\{uv\}$, its decrease is $|uv|-|v|=|u|$.
If $v$ is already there, including the possibility $v=u$, its
decrease is $|uv|$. Termination follows. At termination no distinct
pair has the indicated prefix relation, so the final set is the
required prefix code $D$.

One can also track a decomposition certificate for every original
word explicitly. Initially its certificate is its own generator
symbol. Whenever $uv$ is removed, replace each occurrence of that
generator in all certificates by the symbols for $u$ followed by
$v$, identifying duplicate symbols if $v$ already occurs.
Each certificate remains nonempty and spells the same literal word.
At termination these certificates express every member of $C$
as a nonempty product of members of $D$.
\end{proof}

\begin{lemma}\label{lem:decode}
Let $D=\{d_1,\ldots,d_\ell\}$ be a finite nonempty prefix code
of distinct nonempty words over an alphabet $A$, and let
$Y=\{y_1,\ldots,y_\ell\}$. The morphism
\[
 f:Y^+\longrightarrow A^+,\qquad f(y_i)=d_i,
\]
is injective. Every morphism $h:X^+\longrightarrow A^+$ with
$h(x)\in D^+$ for all $x$ factors as $h=f\circ\sigma$, where
$\sigma:X^+\longrightarrow Y^+$ is nonerasing.
\end{lemma}
\begin{proof}
In two equal nonempty concatenations of codewords, one first
codeword is a prefix of the other. Prefix-freeness makes them equal.
Cancel this common nonempty word and repeat. If exactly one
sequence of codewords were exhausted first, a nonempty concatenation
of nonempty words would equal the empty word, which is impossible.
Both sequences therefore exhaust together with the same successive
codewords, proving injectivity.

For each $x$, take a decomposition of $h(x)$ into members of
$D$ and replace these by their corresponding $y_i$ symbols.
This gives a nonempty word $\sigma(x)$; its extension to the free
semigroup satisfies $h=f\circ\sigma$. In particular the
certificates in Lemma~\ref{lem:prefix} supply such decompositions
when that lemma is used. No injectivity of $h$ is asserted.
\end{proof}

\section{Arbitrary semiring identities and the main theorem}

\begin{lemma}\label{lem:flatsum}
In $F_k$, a finite nonempty sum has nonzero value if and only if
all its summands have the same nonzero value. In that case their
common value is the value of the sum.
\end{lemma}
\begin{proof}
This holds for one summand. Inductively, if the partial sum is zero,
the full sum is zero because $0$ absorbs addition. If it is nonzero,
the next addition remains nonzero precisely when its operands are
the same nonzero element, by~\eqref{eq:flat}. The induction
hypothesis now gives the assertion for all the summands.
\end{proof}

\begin{theorem}\label{thm:identities}
For every $n\ge1$, the ai-semiring $F_{n+1}$ satisfies
$\Id_{\le n}(\B)$.
\end{theorem}
\begin{proof}
Fix $n\ge1$ and put $k=n+1$. Take an arbitrary identity of
$\B$ in at most $n$ variables. Lemma~\ref{lem:normal} writes
it, using the ai-semiring axioms, as
\[
 P=u_1+\cdots+u_p\ \approx\ Q=v_1+\cdots+v_q,
 \qquad p,q\ge1,
\]
with nonempty word summands. The same normalizations hold in
$F_k$ by Lemma~\ref{lem:flat}; the content of each side is
preserved. We prove equality under any assignment $\varphi$ of
the variables to $F_k$.

If both sides have value zero, they agree. Otherwise interchange
the two sides if necessary so that $\varphi(P)=w\ne0$.
Lemma~\ref{lem:flatsum} says that each $u_i$ has value $w$.
Every variable occurring in $P$ must therefore have nonzero value:
it occurs in a word summand, whose value would be zero if that
variable had value zero. Regularity, Lemma~\ref{lem:regular},
gives $\con(P)=\con(Q)$, so all relevant variables have values
in $L_k$. This step includes all assignments with zero values:
if a relevant variable were zero, a nonzero side as assumed here
could not exist.

Regard these variable values as literal words and define the
nonerasing morphism $h$ by $h(x)=\varphi(x)$. Nonzero word
evaluation and Lemma~\ref{lem:factorproduct} give
\begin{equation}\label{eq:commonwords}
 h(u_1)=\cdots=h(u_p)=w
\end{equation}
as literal words. Apply Lemma~\ref{lem:prefix} to the finite
nonempty set $C$ of distinct values $h(x)$ for the relevant
variables. Its cardinality is at most $n$, so the resulting prefix
code $D=\{d_1,\ldots,d_\ell\}\subseteq L_k$ has
$1\le\ell\le n$. Lemma~\ref{lem:decode} gives
\[
 h=f\circ\sigma,\qquad f(y_i)=d_i,
 \qquad Y=\{y_1,\ldots,y_\ell\},
\]
where $f:Y^+\longrightarrow A_k^+$ is injective and $\sigma$
is a nonerasing substitution of the original variables by words in
$Y^+$. Variables with equal original values cause no difficulty.

Apply injectivity of $f$ to~\eqref{eq:commonwords}. It gives
\[
 \sigma(u_1)=\cdots=\sigma(u_p)=t
\]
for a single literal word $t\in Y^+$. Write
$s_j=\sigma(v_j)$. Substitution into the original valid identity
and idempotence of addition yield
\[
 \B\models t\approx s_1+\cdots+s_q.
\]
For each $j$, adding $s_j$ to the right-hand sum leaves it
unchanged, by commutativity and idempotence. Thus
\[
 \B\models s_j+t\approx t,
 \qquad\text{that is,}\qquad
 \B\models s_j\preccurlyeq t.
\]
This is the order direction fixed in Section~2.

These inequalities use at most $|Y|=\ell\le n=k-1$ variables,
the images of all $y_i$ belong to $L_k$, and $f(t)=w\in L_k$.
Lemma~\ref{lem:transfer} therefore gives
\[
 h(v_j)=f(s_j)=w\qquad(1\le j\le q).
\]
Since $w\in L_k$, each $v_j$ indeed has nonzero value $w$
under $\varphi$, by Lemma~\ref{lem:factorproduct}. Hence
$\varphi(Q)=w=\varphi(P)$.

We have covered every assignment, including the possibility that
initially exactly one side might have appeared to be zero: a
nonzero side forces the other to have the same nonzero value.
There is no restriction on word lengths, the lengths of assigned
words, or the finite numbers of summands. This proves the theorem
for every original term identity with the prescribed variable bound.
\end{proof}

\begin{theorem}\label{thm:main}
For every $n\ge1$ there is an infinite ai-semiring $T_n$, generated by at most
$(6n+8)^2$ elements, such that
\[
 T_n\models\Id_{\le n}(\B).
\]
Consequently, $\B$ is inherently nonfinitely based.
\end{theorem}
\begin{proof}
Take $T_n=F_{n+1}$. Lemma~\ref{lem:flat} makes it infinite and
generated by its alphabet $A_{n+1}$, whose cardinality is
\[
 |A_{n+1}|=\bigl(6(n+1)+2\bigr)^2=(6n+8)^2.
\]
Theorem~\ref{thm:identities} gives the asserted identities.

Suppose a locally finite variety $\mathcal V$ of signature
$\{+,\cdot\}$ contained $\B$ and had a finite identity basis
$\Sigma$. Choose $n\ge1$ at least as large as the number of
distinct variables in the union of both sides of each identity of $\Sigma$;
if $\Sigma$ is empty choose $n=1$. Since $\B\in\mathcal V$,
every identity in $\Sigma$ holds in $\B$. The first assertion
therefore gives $T_n\models\Sigma$, so $T_n\in\mathcal V$.
This contradicts local finiteness because $T_n$ is infinite and
finitely generated. The argument concerns every variety in the
stated signature; it does not assume that its members were initially
required to be ai-semirings. Thus $\B$ is inherently nonfinitely
based in precisely that signature.
\end{proof}

\bibliographystyle{amsplain}

\end{document}